\documentclass{amsart} 

\usepackage[english]{babel}
\usepackage[latin1]{inputenc}
\usepackage[autostyle, english = american]{csquotes} 

\usepackage{amsmath}
\usepackage{amsthm}
\usepackage{amssymb}
\usepackage{tikz}

\usepackage{imakeidx} 
\usepackage{appendix}
\usepackage{tikz-cd}
\usepackage{verbatim}
\usepackage{enumitem}
\usepackage{multicol}

\usepackage[backref=page]{hyperref}
\usepackage{cleveref}

\hypersetup{colorlinks=true,linkcolor=blue,anchorcolor=blue,citecolor=blue}

\usepackage{adjustbox}

\usepackage{mathtools}

\usepackage{amsfonts}

\newcommand{\overbar}[1]{\mkern 1.5mu\overline{\mkern-3mu#1\mkern-1.5mu}\mkern 1.5mu}

\newtheorem{thm}{Theorem}[section]
\newtheorem{prop}[thm]{Proposition}
\newtheorem{lemma}[thm]{Lemma}
\newtheorem{cor}[thm]{Corollary}
\newtheorem{conj}[thm]{Conjecture}
\newtheorem{question}[thm]{Question}

\theoremstyle{definition}

\newtheorem{example}[thm]{Example}
\newtheorem{claim}[thm]{Claim}

\theoremstyle{remark}
\newtheorem{rmk}[thm]{Remark}
\numberwithin{equation}{section}

\usepackage{relsize}
\usepackage[bbgreekl]{mathbbol}
\usepackage{amsfonts}
\DeclareSymbolFontAlphabet{\mathbb}{AMSb} 
\DeclareSymbolFontAlphabet{\mathbbl}{bbold}

\newcommand{\Q}{\mathbb Q}
\newcommand{\F}{\mathbb F}

\newcommand{\Z}{\mathbb Z}

\renewcommand{\phi}{\varphi}

\newcommand{\on}[1]{\operatorname{#1}}
\newcommand{\ang}[1]{\left \langle{#1}\right \rangle}

\title[Galois representations modulo $p$ that do not lift modulo $p^2$]{Galois representations modulo $p$ that do not lift modulo $p^2$}

\author{Alexander Merkurjev}
\address{Department of Mathematics\\
	University of California\\
	Los Angeles, CA 90095 \\United States of America}
\email{merkurev@math.ucla.edu}

\author{Federico Scavia}
\address{CNRS\\
	Institut Galil\'ee\\
	Universit\'e Sorbonne Paris Nord\\
	99 avenue Jean-Baptiste Cl\'ement, 93430\\ 
	Villetaneuse, France}
\email{scavia@math.univ-paris13.fr}

\makeatletter
\@namedef{subjclassname@2020}{\textup{2020} Mathematics Subject Classification}
\makeatother


\date{October 16, 2024}

\begin{document}

	\begin{abstract}
		For every finite group $H$ and every finite $H$-module $A$, we determine the subgroup of negligible classes in $H^2(H,A)$, in the sense of Serre, over fields with enough roots of unity. As a consequence, we show that for every odd prime $p$, every integer $n\geq 3$, and every field $F$ containing a primitive $p$-th root of unity, there exists a continuous $n$-dimensional mod $p$ representation of the absolute Galois group of $F(x_1,\dots,x_p)$ which does not lift modulo $p^2$. This answers a question of Khare and Serre, and disproves a conjecture of Florence.
	\end{abstract}
	
	\maketitle
	
	\section{Introduction}
	
	\subsection{The profinite inverse Galois problem} A central issue in modern Galois theory is the profinite inverse Galois problem, which asks how to characterize absolute Galois groups of fields among profinite groups. While an answer to this question is not known, even conjecturally, several necessary conditions for a profinite group to qualify as an absolute Galois group have been established.
	
	The most classical result in this direction is due to Artin and Schreier \cite{artin1924kennzeichnung,artin1927kennzeichnung}, who proved that every non-trivial finite subgroup of an absolute Galois group is cyclic of order $2$. A deeper necessary condition is the Bloch--Kato conjecture, now a theorem due to Voevodsky and Rost \cite{haesemeyer2019norm}, which in particular implies that the mod $p$ cohomology ring of an absolute Galois group of a field containing a primitive $p$-th root of unity is generated in degree $1$ with relations in degree $2$. As a more recent example, the Massey vanishing conjecture of Min\'{a}\v{c} and T\^{a}n \cite{minac2017triple} predicts that all non-empty Massey products of $n\geq 3$ elements in the mod $p$ Galois cohomology of fields contain the zero element. This conjecture has been proved in several cases \cite{hopkins2015splitting,minac2015kernel,minac2016triple,efrat2017triple,harpaz2019massey,merkurjev2022degenerate,merkurjev2023massey}, and each of these theorems provides new restrictions on the cohomology of absolute Galois groups of fields. 
	
	In this article, we investigate restrictions to the profinite inverse Galois problem coming from the embedding problem with abelian kernel.

	\subsection{The embedding problem with abelian kernel} 
	Let $F$ be a field, let $H$ be a finite group, and let $A$ be a finite $H$-module. Consider a group extension
	\begin{equation}\label{group-extension}
		\begin{tikzcd}
			0 \arrow[r]  & A \arrow[r] & G \arrow[r] & H \arrow[r] & 1,
		\end{tikzcd}
	\end{equation}
	where the $H$-module action on $A$ coincides with the action induced by the conjugation of $G$ on $A$. We say that the embedding problem for (\ref{group-extension}) over $F$ has a positive solution if for every field extension $K/F$, letting $\Gamma_K$ denote the absolute Galois group of $K$, every continuous homomorphism $\rho\colon \Gamma_K\to H$ lifts to a continuous homomorphism $\tilde{\rho}\colon \Gamma_K\to G$:
	\[
	\begin{tikzcd}
		&&& \Gamma_K \arrow[d,"\rho"] \arrow[dl,dotted,swap,"\tilde{\rho}"]   \\
		0 \arrow[r]  & A \arrow[r] & G \arrow[r] & H \arrow[r] & 1.
	\end{tikzcd}
	\]
	To prove that the embedding problem for (\ref{group-extension}) over $F$ has a positive solution, it is enough to lift a single continuous homomorphism $\Gamma_K\to H$, corresponding to a \enquote{generic} Galois field extension $L/K$ over $F$ with Galois group $H$; see \Cref{connection}. 
	
	This is a lifting problem. The terminology \enquote{embedding problem} is classical, and is motivated by the fact that, letting $L/K$ be a Galois $H$-algebra corresponding to $\rho$, a lifting $\tilde{\rho}$ of $\rho$ exists if and only if there exists a Galois $G$-algebra $M/K$ such that $M^A=L$ (that is, $L$ embeds in $M$ as the $A$-invariant subalgebra). We refer the reader to the monograph \cite{ishkhanov1997embedding} for further information on this classical and vast subject.
	
	\subsection{Lifting Galois representations} Extensions of the form (\ref{group-extension}) are of course very general. Before the present work, the embedding problem with abelian kernel was open even for some very basic extensions. Consider the following question of Khare and Serre (cf. \cite[Question 1.1]{khare2020liftable}).
	
	\begin{question}[Khare, Serre]\label{question}
		Let $F$ be a field. Does every continuous homomorphism $\Gamma_F\to \on{GL}_n(\F_p)$ lift to a continuous homomorphism $\Gamma_F\to \on{GL}_n(\Z/p^2\Z)$?
	\end{question} 
	In other words, does every $n$-dimensional continuous mod $p$ representation of $\Gamma_F$ lift modulo $p^2$? \Cref{question} is equivalent to the embedding problem for the extension
	\begin{equation}\label{intro-gln-ext}
		\begin{tikzcd}
			0 \arrow[r] & \mathfrak{gl}_n(\F_p) \arrow[r] & \on{GL}_n(\Z/p^2\Z) \arrow[r] & \on{GL}_n(\F_p) \arrow[r] & 1
		\end{tikzcd}
	\end{equation}
	over all fields.
	\Cref{question} has an affirmative answer for $n=1$ and any prime $p$, because the surjection $(\Z/p^2\Z)^\times\to \F_p^\times$ is split. When $n=2$, \Cref{question} also has a positive answer for every prime $p$, by a theorem of Khare and Serre \cite{khare1997base}. More precisely, Khare proved the theorem for number fields, and Serre observed that Khare's argument involved only Kummer theory, and hence generalized to an arbitrary field; see \cite[Remark 2 p. 392]{khare1997base}. To our knowledge, \Cref{question} originated from this result. De Clercq and Florence \cite{declercq2022lifting} answered \Cref{question} affirmatively when $3\leq n\leq 4$ and $p=2$.
	
	Lifting methods for $2$-dimensional representations are well developed over number fields, following the work of Ramakrishna \cite{ramakrishna} and Khare--Wintenberger \cite{khare2009serre}. B\"ockle \cite{bockle2003lifting} obtained a positive answer to \Cref{question} when $F$ is a local field, and also when $F$ is a global field and the image of $\Gamma_F\to \on{GL}_n(\F_p)$ is sufficiently large. When $F$ is a number field containing a primitive root of unity of order $p^2$, all $3$-dimensional mod $p$ representations of $\Gamma_F$ lift modulo $p^2$, by Khare--Larsen \cite{khare2020liftable}.
	
	Let $B_n\subset \on{GL}_n$ be the Borel subgroup of upper-triangular matrices. Consider the following conjecture of Florence (cf. \cite{florence2020smooth}).
	
	\begin{conj}[Florence]\label{florence-conj}
		Let $p$ be a prime number, let $F$ be a field containing a primitive $p^2$-th root of unity, and let $n\geq 1$ be an integer. Every continuous homomorphism $\Gamma_F\to B_n(\F_p)$ lifts to a continuous homomorphism $\Gamma_F\to B_n(\Z/p^2\Z)$.	
	\end{conj}
	In other words, every length $n$ complete flag of continuous representations of $\Gamma_F$ modulo $p$ should lift to a complete flag modulo $p^2$. \Cref{florence-conj} is equivalent to the embedding problem for the extension
	\begin{equation}\label{intro-bn-ext}
		\begin{tikzcd}
			0 \arrow[r] & \mathfrak{b}_n(\F_p) \arrow[r] & B_n(\Z/p^2\Z) \arrow[r] & B_n(\F_p) \arrow[r] & 1
		\end{tikzcd}
	\end{equation}
	over fields containing a primitive $p^2$-th root of unity. By a restriction-corestriction argument, \Cref{florence-conj} implies a positive answer to \Cref{question}. The original form of \Cref{florence-conj} did not assume the existence of a primitive $p^2$-th root of unity, but this version is false, as demonstrated by a recent counterexample of Florence \cite{florence2024triangular}. (As observed by Florence, this is irrelevant for potential applications of \Cref{florence-conj}; see below.) In Florence's example, the prime $p$ is odd, $n=3$, and $F=\Q(\!(t)\!)$.
	
	\Cref{florence-conj} is motivated by ongoing work of De Clercq and Florence on the Bloch--Kato conjecture \cite{declercq2017lifting, declercq2022lifting, declercqc2020smooth-1, florence2020smooth, declercq2020smooth-3}. Voevodsky and Rost proved the Bloch--Kato conjecture using motivic cohomology and triangulated categories of motives. As observed by De Clercq and Florence, \Cref{florence-conj} implies the surjectivity assertion of the Bloch--Kato conjecture for \emph{all} fields. It is well known that this in turn implies the full Bloch--Kato conjecture by elementary arguments; see \cite[Th\'eor\`eme 0.1]{gille2007symbole}. De Clercq and Florence aimed to prove \Cref{florence-conj} using only Hilbert's Theorem 90 and Kummer theory, in a precise sense \cite[Conjecture 14.25]{declercq2017lifting}, thus achieving an elementary proof of the Bloch--Kato conjecture. Conti--Demarche--Florence \cite{conti2024lifting} proved \Cref{florence-conj} for all local fields $F$, without assumptions on roots of unity.

	In this article, as an application of our main theorem, \emph{we negatively answer \Cref{question} and disprove \Cref{florence-conj}, even over fields containing all roots of unity}.
	
	\subsection{The main theorem} The main result of this article is the determination, in group cohomology terms, of all the extensions (\ref{group-extension}) for which the embedding problem has positive solution, in the case when $F$ contains a primitive root of unity of order $e(A)e(H)$. (Here and in what follows, $e(-)$ denotes the exponent of a finite group, that is, the least common multiple of the orders of its elements.) The existence of such a uniform description was surprising to us: extensions (\ref{group-extension}) are very general, and, as discussed above, until now the answer to the embedding problem with abelian kernel was not known even in basic examples such as (\ref{intro-gln-ext}) and (\ref{intro-bn-ext}).
	
	To state our result, let $\alpha\in H^2(H,A)$ be the class of (\ref{group-extension}), and let $F$ be a field. Following Serre, we say that $\alpha$ is \emph{negligible over $F$} if, for every field extension $K/F$ and every continuous homomorphism $\rho\colon \Gamma_K\to H$, we have $\rho^*(\alpha)=0$ in $H^2(\Gamma_K,A)$. It is easy to see that $\alpha$ is negligible over $F$ if and only if the embedding problem for (\ref{group-extension}) has a positive solution over $F$; see \Cref{connection}. Thus \Cref{question} may be rephrased as: Is the class of (\ref{intro-gln-ext}) in $H^2(\on{GL}_n(\F_p),\mathfrak{gl}_n(\F_p))$ negligible over $F$? Similarly, \Cref{florence-conj} predicts that the class of (\ref{intro-bn-ext}) in $H^2(B(\F_p),\mathfrak{b}_n(\F_p))$ is negligible over every field $F$ containing a primitive root of unity of order $p^2$.
	
	\begin{thm}\label{thm-negligible}
		Let $H$ be a finite group, let $A$ be a finite $H$-module, and let $F$ be a field
		containing a primitive root of unity of order $e(A) e(H)$. Then the subgroup of negligible classes in $H^2(H,A)$ is generated by the images of the maps
		\[A^{H'}\otimes H^2(H',\Z)\xrightarrow{\cup} H^2(H',A)\xrightarrow{\on{cor}} H^2(H, A),\]
		where $H'$ ranges over all subgroups of $H$.
	\end{thm}
	In particular, if $F$ has enough roots of unity, the subgroup of negligible classes over $F$ is independent of $F$, and it admits a description in pure group cohomology terms. When $A$ acts trivially on $H$, the subgroup of negligible classes was determined by Gherman and the first author in \cite{gherman2022negligible}, without assumptions on primitive roots of unity. Their result implies that the assumption on roots of unity in \Cref{thm-negligible} is sharp. 
	
	\subsection{Non-liftable Galois representations} 
	In addition to its conceptual significance, the description of $H^2(H,A)_{\on{neg},F}$ given in \Cref{thm-negligible} is easy to use in concrete examples. We illustrate this with the following application.
	
	\begin{thm}\label{thm-galois-rep}
		Let $p$ be an odd prime, let $F$ be a field containing a primitive $p$-th root of unity, and let $K\coloneqq F(x_1,\dots,x_p)$, where the $x_i$ are algebraically independent variables over $F$. For every integer $n\geq 3$, there exists a continuous homomorphism $\Gamma_K\to \on{GL}_n(\F_p)$ which does not lift to a continuous homomorphism $\Gamma_K\to \on{GL}_n(\Z/p^2\Z)$.
	\end{thm}
	In particular, (\ref{intro-gln-ext}) and (\ref{intro-bn-ext}) are not negligible over \emph{every} field $F$ of characteristic different from $p$. (Conversely, if $\on{char}(F)=p$, then $H^2(F,M)=0$ for every finite $p$-torsion $\Gamma_F$-module $M$, and hence (\ref{intro-gln-ext}) and (\ref{intro-bn-ext}) are negligible over $F$.)
	
	\Cref{thm-galois-rep} shows that \Cref{question} has a negative answer, and hence that \Cref{florence-conj} is false, even over fields containing all roots of unity. 
	
	\subsection{Contents}
	We conclude this introduction with a description of the contents of each section. In \Cref{section-2}, we define and review the basic properties of degree $2$ negligible cohomology classes. In \Cref{section-3}, we describe the subgroup of degree $2$ negligible classes as the image of a certain map $\gamma$, closely related to the transgression map in group cohomology, and we analyze the map $\gamma$. In \Cref{section-4}, we use the work of the previous section to prove \Cref{thm-negligible}. Finally, in \Cref{section-5} we prove that (\ref{intro-gln-ext}) and (\ref{intro-bn-ext}) are not negligible (\Cref{thm-gl3}): this already gives a negative answer to \Cref{question} and disproves \Cref{florence-conj}. A slight refinement of the argument, combined with the positive solution to the Noether problem for the group of rank $3$ upper unitriangular matrices $U_3(\F_p)$, due to Chu and Kang \cite{chu2001rationality}, then proves \Cref{thm-galois-rep}.

	\subsection*{Notation}
	
	For a profinite group $\Gamma$ and a discrete $\Gamma$-module $A$, we write $H^i(\Gamma,A)$ for the $i$-th cohomology group of $\Gamma$ with coefficients on $A$. For a field $F$, we let $F^{\on{sep}}$ be a separable closure of $F$, we let $\Gamma_F\coloneqq \on{Gal}(F^{\on{sep}}/F)$ be the absolute Galois group of $F$, and for every discrete $\Gamma_F$-module $A$ we set $H^i(F,A)\coloneqq H^i(\Gamma_F,A)$. If $G$ is a finite group, we let $e(G)$ be the exponent of $G$, that is, the least common multiple of the orders of the elements of $G$.

	\section{Negligible cohomology}\label{section-2}
	
	Let $H$ be a finite group, let $A$ be an $H$-module, let $F$ be a field, and let $d$ be a non-negative integer. An element $\alpha\in H^d(H,A)$ is called
	\emph{negligible over $F$} if for every field extension $K/F$ and every continuous group homomorphism $\Gamma_K\to H$,
	the induced homomorphism $H^d(H,A)\to H^d(K,A)$ takes $\alpha$ to zero. (Here $A$ is viewed as a discrete $\Gamma_K$-module via the homomorphism $\Gamma_K\to H$.) This definition is due to Serre; see \cite[\S 26 p. 61]{garibaldi2003cohomological}. The negligible elements over $F$ form a subgroup
	\[
	H^d(H,A)_{\on{neg}, F} \subset H^d(H,A).
	\]

	Suppose now that $d=2$, and let $\alpha\in H^2(H,A)$. Then $\alpha$ represents a group extension
	\begin{equation}\label{extension-eq}
		\begin{tikzcd}
			0 \arrow[r]  & A \arrow[r] & G \arrow[r] & H \arrow[r] & 1,
		\end{tikzcd}
	\end{equation}
	where the $H$-action on $A$ induced by the conjugation $G$-action coincides with the $H$-module action. Conversely, every extension (\ref{extension-eq}) is represented by an element in $H^2(H,A)$.

	Let $V$ be a (finite-dimensional) faithful representation of $H$ over $F$. Then $F(V)/F(V)^H$ is a Galois field extension with Galois group $H$.
	It is called a \emph{generic} Galois field extension with Galois group $H$. The word ``generic" is explained by the next proposition.	
	
	\begin{prop}\label{connection}
		Let $H$ be a finite group, let $A$ be an $H$-module, and consider the class $\alpha\in H^2(H,A)$ of an extension (\ref{extension-eq}). Let $V$ be a faithful $H$-representation over $F$, fix an embedding of the Galois extension $F(V)/F(V)^H$ into a separable closure of $F(V)^H$, and let $\rho\colon \Gamma_{F(V)^H}\to H$ be the corresponding continuous homomorphism. The following are equivalent.
		\begin{enumerate}
			\item The class $\alpha$ is negligible over $F$.
			\item For every field extension $K/F$, every continuous homomorphism $\Gamma_K\to H$ lifts to $G$.
			\item The continuous homomorphism $\rho$ lifts to $G$.
		\end{enumerate}
		Furthermore,
		\[H^2(H,A)_{\on{neg}, F}=\on{Ker}[H^2(H,A)\xrightarrow{\on{inf}} H^2(F(V)^H,A)].\]
	\end{prop}
	
	\begin{proof}
		See \cite[Proposition 2.1, Corollary 2.2]{gherman2022negligible}.
	\end{proof}
	
	\begin{example}[Kummer theory]\label{kummer-theory}
		Let $m,n\geq 1$ be integers, and view $\Z/n\Z$ as a $\Z/m\Z$-module with trivial action. Then every class in $H^2(\Z/m\Z,\Z/n\Z)$ is negligible over every field $F$ containing a primitive root of unity of order $mn$. Indeed, the group $H^2(\Z/m\Z,\Z/n\Z)$ is generated by the class $\alpha$ of the extension
		\[
		\begin{tikzcd}
			0 \arrow[r]  & \Z/n\Z \arrow[r] & \Z/mn\Z \arrow[r,"\pi"] & \Z/m\Z \arrow[r] & 0,
		\end{tikzcd}
		\]
		where $\pi$ is the reduction map. It thus suffices to show that $\alpha$ is negligible over $F$. For this, we observe that, because $F$ contains a primitive root of unity of order $mn$, by Kummer theory the map $\pi_*\colon H^1(K,\Z/mn\Z)\to H^1(K,\Z/m\Z)$ is (non-canonically) identified with the quotient map $F^\times/F^{\times mn}\to F^\times/F^{\times m}$, and hence in particular it is surjective. In other words, every continuous homomorphism $\Gamma_F\to \Z/m\Z$ lifts to $\Z/mn\Z$. By \Cref{connection}, this is equivalent to the assertion that $\alpha$ is negligible over $F$.
	\end{example}
	
	\begin{lemma}\label{props}
		Let $F$ be a field. 
		\begin{enumerate}
			\item For every finite group $H$ and every $H$-module homomorphism $B\to A$, the induced map
			$H^2(H,B)\to H^2(H,A)$ takes the subgroup $H^2(H,B)_{\on{neg},F}$ into $H^2(H,A)_{\on{neg},F}$.
			\item For every homomorphism of finite groups $H'\to H$ and every $H$-module $A$, the pullback map
			$H^2(H,A) \to H^2(H',A)$ takes the subgroup $H^2(H,A)_{\on{neg},F}$ into $H^2(H',A)_{\on{neg},F}$.
			\item For every finite group $H$, every subgroup $H'$ of $H$, and every $H$-module $A$, the corestriction $H^2(H',A) \to H^2(H,A)$ takes the subgroup $H^2(H',A)_{\on{neg},F}$ into $H^2(H,A)_{\on{neg},F}$.
			\item For every finite group $H$, every subgroup $H'$ of $H$, every finite $H$-module $A$, if $[H:H']$ is prime to $e(A)$, then a class $\alpha \in H^2(H,A)$ is negligible if and only if its restriction to $H^2(H',A)$ is negligible.
			\item For every finite group $H$, every $H$-module $A$, and every field extension $L/F$, we have $H^2(H,A)_{\on{neg},F}\subset H^2(H,A)_{\on{neg},L}$.
		\end{enumerate}
	\end{lemma}
	
	\begin{proof}
		(1), (2) and (5) are proved in \cite[Proposition 2.3]{gherman2022negligible}.
		For (3), consider a faithful $H$-representation $V$ over $F$.
		The square
		\[
		\begin{tikzcd}
			H^2(H',A) \arrow[d,"\on{cor}"] \arrow[r,"\on{inf}"] &  H^2(F(V)^{H'},A)  \arrow[d,"\on{cor}"] \\
			H^2(H,A) \arrow[r,"\on{inf}"]  &  H^2(F(V)^{H},A)
		\end{tikzcd}
		\]
		is commutative by \cite[Proposition 1.5.5]{neukirch2008cohomology}. Now \Cref{connection} implies (3). Finally, (4) follows from (2), (3) and a restriction-corestriction argument.
	\end{proof}
	
	For a finite Galois extension $L/K$ with Galois group $H$, the Lyndon--Hochschild--Serre spectral sequence
	\[
	E_2^{ij}\coloneqq H^i(H, H^j(L,A))\Longrightarrow H^{i+j}(K,A)
	\]
	gives a transgression map
	\begin{equation}\label{tg}\on{tg}\colon H^1(L,A)^H\to H^2(H,A)\end{equation}
	fitting into a short exact sequence
	\begin{equation}\label{tg-inf-2}H^1(L,A)^H\xrightarrow{\on{tg}} H^2(H,A)\xrightarrow{\on{inf}}H^2(L,A).
	\end{equation}
	
	\begin{cor}\label{connection-tg}
		Under the assumptions of \Cref{connection}, we have
		\[H^2(H,A)_{\on{neg}, F}=\on{Im}[H^1(F(V),A)^H\xrightarrow{\on{tg}} H^2(H,A)].\]
	\end{cor}
	
	\begin{proof}
		Combine \Cref{connection} and (\ref{tg-inf-2}).
	\end{proof}
	
	\begin{lemma}\label{cor-tg}
		Let $H'\subset H$ be a subgroup. We have a commutative square
		\[
		\begin{tikzcd}
			H^1(L,A)^{H'} \arrow[r,"\on{tg}"] \arrow[d,"N_{H/H'}"] & H^2(H',A) \arrow[d,"\on{cor}^{H'}_H"]  \\
			H^1(L,A)^H \arrow[r,"\on{tg}"] & H^2(H,A).
		\end{tikzcd}
		\]
	\end{lemma}
	
	\begin{proof}
		Consider the short exact sequence $1\to \Gamma_L\to \Gamma_K\to H\to 1$, let $\Gamma_L\to \Delta$ be the quotient by the closure of the derived subgroup of $\Gamma_L$, and let $u\in H^2(H,\Delta)$ be the class of the pushout extension
		\begin{equation}\label{cor-tg-1}
			\begin{tikzcd}
				1 \arrow[r] & \Delta \arrow[r] & \Gamma \arrow[r] & H \arrow[r] & 1.        
			\end{tikzcd}
		\end{equation}
		By \cite[Theorem 2.4.4]{neukirch2008cohomology}, the map $\on{tg}\colon H^1(L,A)^H\to H^2(H,A)$ is given by $x\mapsto -u\cup x$.
		Let $u'\coloneqq \on{res}^H_{H'}(u)\in H^2(H',\Delta)$. Then $u'$ represents the pullback
		\[
		\begin{tikzcd}
			1 \arrow[r] & \Delta \arrow[r] & \Gamma' \arrow[r] & H' \arrow[r] & 1    
		\end{tikzcd}
		\]
		of (\ref{cor-tg-1}) to $H'$. Thus, invoking \cite[Theorem 2.4.4]{neukirch2008cohomology} again, the transgression map $\on{tg}\colon H^1(L,A)^{H'}\to H^2(H',A)$ is given by $x\mapsto -u'\cup x$. Combining this with the projection formula \cite[Proposition 1.5.3(iv)]{neukirch2008cohomology}, for all $x\in H^1(L,A)^{H'}$ we obtain
		\[\on{cor}^{H'}_H(\on{tg}(x))=-\on{cor}^{H'}_H(u\cup x)=-u'\cup N_{H/H'}(x)=\on{tg}(N_{H/H'}(x))\]
		in $H^2(H,A)$, as desired.
	\end{proof}
	
	\section{The map \texorpdfstring{$\gamma$}{gamma}}\label{section-3}
	Let $e,n\geq 1$ be integers, let $H$ be a finite group of exponent dividing $e$, let $A$ be a finite $H$-module such that $nA=0$, and let $L/K$ be a Galois field extension with Galois group $H$. We assume that $K$ contains a primitive root of unity of order $ne$; in particular, $ne$ is invertible in $K$. We regard $\mu_{ne}$ as an $H$-submodule of $K^\times$; note that $H$ acts trivially on $\mu_{ne}$. We define the $H$-module $A(-1)\coloneqq \on{Hom}(\mu_n,A)$, and we identify $A(-1)\otimes \mu_n$ with $A$ via the isomorphism of $H$-modules $\varphi\otimes \zeta\mapsto \varphi(\zeta)$.
	
	We have an isomorphism of $H$-modules \[
	\iota\colon A(-1)\otimes L^\times \xrightarrow{\sim} A(-1)\otimes H^1(L,\mu_n)\xrightarrow{\sim} H^1(L,A).\]
	Here, the first map is obtained by tensorization of the Kummer isomorphism $L^\times/L^{\times n}\xrightarrow{\sim} H^1(L,\mu_n)$ with $A(-1)$. The second map is the cup product map $H^0(L,A(-1))\otimes H^1(L,\mu_n)\to H^1(L,A)$: to see that it is an isomorphism, by additivity in $A$ one is immediately reduced to the case when $A$ is cyclic, in which case the result follows from the computation of the cohomology of finite cyclic groups \cite[Proposition 1.7.1]{neukirch2008cohomology}. 
	
	We consider the composite map
	\begin{equation}\label{gamma}
		\gamma\colon (A(-1)\otimes L^\times)^H \xrightarrow{\iota} H^1(L,A)^H\xrightarrow{\on{tg}} H^2(H,A),
	\end{equation}
	where $\on{tg}$ is the transgression map of (\ref{tg}). 
	
	We let $\partial_1\colon (L^\times/\mu_{e})^H\to H^1(H,\mu_{e})$ be the connecting homomorphism associated to the short exact sequence of $H$-modules
	\begin{equation}\label{delta-1}
		\begin{tikzcd}
			1 \arrow[r] & \mu_e \arrow[r] & L^\times \arrow[r] & L^\times/\mu_e \arrow[r] & 1. 
		\end{tikzcd}
	\end{equation}
	
	We also let $\partial_2\colon H^1(H,\mu_{e})\to H^2(H,\mu_n)$ be the connecting homomorphism associated to the short exact sequence of $H$-modules
	\begin{equation}\label{delta-2}
		\begin{tikzcd}
			1 \arrow[r] & \mu_n \arrow[r] & \mu_{ne} \arrow[r] & \mu_e \arrow[r] & 1. 
		\end{tikzcd}
	\end{equation}
	Let $f\in L^\times$ be such that $f\mu_{e}\in (L^\times/\mu_{e})^H$. Consider the map $\Z/n\Z\to L^\times/L^{\times n}$ taking $1$ to $f L^{\times n}$. We claim that this map is $H$-equivariant: indeed, by assumption, $\mu_{ne}\subset L^\times$, and hence for all $h\in H$ we have $(h-1)f\in \mu_{e}\subset (\mu_{n e})^n\subset L^{\times n}$,
	from which the claim follows. Tensorization with $A(-1)$ and passage to $H$-invariants now yields, for every such $f$, a map $\beta_f\colon A(-1)^H\to (A(-1)\otimes L^\times)^H$. Concretely, $\beta_f(a)=a\otimes f$ for all $a\in A(-1)^H$. This defines a homomorphism
	\begin{equation}\label{beta}\beta\colon A(-1)^H\otimes (L^\times/\mu_{e})^H\to (A(-1)\otimes L^\times)^H,\qquad a\otimes f \mu_{e}\mapsto \beta_f(a)=a\otimes f.
	\end{equation}
	
	\begin{prop}\label{diagram}
		Let $e,n\geq 1$ be integers, let $H$ be a finite group of exponent dividing $e$, let $A$ be a finite $H$-module such that $nA=0$, and let $L/K$ be a Galois field extension with Galois group $H$, such that $K$ contains a primitive root of unity of order $ne$. Then the diagram
		\[
		\begin{tikzcd}
			A(-1)^H\otimes (L^\times/\mu_{e})^H  \arrow[d,"\on{id}\otimes\partial_1"] \arrow[r,"\beta"] &  (A(-1)\otimes L^\times)^H \arrow[dd,"\gamma"]  \\
			A(-1)^H\otimes H^1(H,\mu_{e}) \arrow[d,"\on{id}\otimes\partial_2"]  \\
			A(-1)^H\otimes H^2(H,\mu_n) \arrow[r,"\cup"] & H^2(H,A)
		\end{tikzcd}
		\]
		is anti-commutative.
	\end{prop}
	
	For the proof of \Cref{diagram}, the key point is the following observation.
	
	\begin{lemma}\label{galois-algebras}
		Let $e,n\geq 1$ be integers, let $H$ be a finite group of exponent dividing $e$, let $L/K$ be a Galois field extension with Galois group $H$, such that $K$ contains a primitive root of unity of order $ne$. Let $f\in L^\times$ be such that $f\mu_e\in (L^\times/\mu_e)^H$. Let $\rho_f\colon \Gamma_L\to \mu_n$ be the continuous homomorphism corresponding to $f L^\times$ under the Kummer isomorphism $L^\times/L^{\times n}\xrightarrow{\sim} H^1(L,\mu_n)$. Then there exists a commutative diagram with exact rows
		\[
		\begin{tikzcd}
			1 \arrow[r]  & \Gamma_L \arrow[r]\arrow[d,"\rho_f"] & \Gamma_K \arrow[r]\arrow[d]  & H \arrow[r]\arrow[d,"\partial_1(f\mu_e)"]  & 1 \\
			1 \arrow[r] & \mu_n \arrow[r] & \mu_{ne} \arrow[r] & \mu_e \arrow[r] & 1.
		\end{tikzcd}
		\]
	\end{lemma}
	
	\begin{proof}
		Consider the \'etale $F$-algebra $M\coloneqq L[t]/(t^n-f)$, and let $f^{1/n}\coloneqq t$. For every $h\in H$, we have $(h-1)f=\partial_1(f\mu_e)(h)\in \mu_e=(\mu_{ne})^n$, and hence we may choose $c_h\in \mu_{ne}$ such that $(h-1)f=(c_h)^n$. We can define a $K$-algebra automorphism $\tilde{h}\in \on{Aut}_K(M)$ such that $\tilde{h}(f^{1/n})=c_h f^{1/n}$ and $\tilde{h}$
		acts on $L$ as $h$. We also have $\mu_n$ acting naturally on $M$ over $L$. We define $G\subset \on{Aut}_K(M)$ as the subset of elements of the form $\zeta\tilde{h}$, where $\zeta\in\mu_n$ and $h\in H$. Observe that all the $\tilde{h}$ commute with elements in $\mu_n$, and that for all $h_1,h_2,h_3\in H$ such that $h_1h_2=h_3$, the element $\tilde{h}_1\tilde{h}_2$ differs from $\tilde{h}_3$ by an element of $\mu_n$. Thus $G$ is in fact a subgroup of $\on{Aut}_K(M)$, and we have a central exact sequence 
		\[
		\begin{tikzcd}
			1 \arrow[r] &  \mu_n \arrow[r] & G \arrow[r] & H \arrow[r] & 1,
		\end{tikzcd}
		\]
		where the map $G\to H$ sends $\tilde{h}$ to $h$ for every $h\in H$. Then $M/K$ is a Galois $G$-algebra, that is, $M^G=K$ and $|G|=[M:K]$; see \cite[Definition 18.15]{knus1998book}.
		
		Recall that the non-abelian cohomology group $H^1(K,G)=\on{Hom}(\Gamma_K,G)/\sim$, where $\sim$ is the equivalence relation determined by conjugation by elements of $G$, classifies isomorphism classes of Galois $G$-algebras over $K$; see \cite[Example 28.15]{knus1998book}. Let $\Gamma_K\to G$ be a continuous homomorphism corresponding to $M/K$.
		We obtain a commutative diagram with exact rows
		\begin{equation}\label{galois-algebras-diag1}
			\begin{tikzcd}
				1 \arrow[r]  & \Gamma_L \arrow[r]\arrow[d,"\rho_f"] & \Gamma_K \arrow[r]\arrow[d]  & H \arrow[r]\arrow[d,equal]  & 1 \\
				1 \arrow[r] & \mu_n \arrow[r] & G \arrow[r] & H \arrow[r] & 1,
			\end{tikzcd}
		\end{equation}
		where the homomorphism $\Gamma_K \to H$ is defined as the composite $\Gamma_K\to G\to H$. The commutativity of the left square follows from the fact that, by Kummer theory, the homomorphism $\rho_f$ corresponds to the Galois $\mu_n$-algebra $M/L$ under the isomorphism of \cite[Example 28.15]{knus1998book}. Similarly to the definition of $\partial_1$, we let $\tilde{\partial}_1\colon (M^{\times}/\mu_{ne})^G\to H^1(G,\mu_{ne})$ be the connecting homomorphism associated to the exact sequence of $G$-modules 
		\[
		\begin{tikzcd}
			1 \arrow[r] & \arrow[r]  \mu_{ne} & \arrow[r] M^\times  & M^{\times}/\mu_{ne} \arrow[r] & 1,
		\end{tikzcd}
		\]
		and we let $\chi\colon G \to \mu_{ne}$ be the image of $f^{1/n}$ under $\tilde{\partial}_1$. In formulas, we have $(g-1)f^{1/n}=\chi(g)$ for every $g\in G$. We obtain a commutative diagram with exact rows
		\begin{equation}\label{galois-algebras-diag2}
			\begin{tikzcd}
				1 \arrow[r] & \mu_n \arrow[r]  \arrow[d,equal]   &  G  \arrow[r]  \arrow[d,"\chi"] &  H  \arrow[r]  \arrow[d,"\partial_1(f\mu_e)"] & 1 \\
				1 \arrow[r] & \mu_n \arrow[r]    &  \mu_{n e} \arrow[r] &  \mu_{e}  \arrow[r] & 1.
			\end{tikzcd}
		\end{equation}
		Combining (\ref{galois-algebras-diag1}) and (\ref{galois-algebras-diag2}) yields the desired diagram.
	\end{proof}
	
	\begin{proof}[Proof of \Cref{diagram}]
		We first reduce to the case where $A=\mu_n$ with trivial $H$-action. For this, fix $f\in L^\times$ such that $f\mu_{e}\in (L^\times/\mu_{e})^H$. We must show that the composite $\gamma\circ\beta_f$ sends $a\in A(-1)^H$ to $-a\cup (\partial_2\partial_1)(f\mu_e)$. Let $a\in A(-1)^H$, and let $\psi_a\colon \Z/n\Z\to A(-1)$ be the homomorphism taking $1$ to $a$. Since $\beta_f$ and $\gamma$ are functorial in $A$, we have a commutative diagram
		\[
		\begin{tikzcd}
			\Z/n\Z \arrow[r,"\beta_f"] \arrow[d,"\psi_a"] & ((\Z/n\Z)\otimes L^\times)^H \arrow[r,"\gamma"] \arrow[d,"\psi_a\otimes\on{id}"] & H^2(H,\mu_n) \arrow[d,"(\psi_a)_*"]\\
			A(-1)^H \arrow[r,"\beta_f"] & (A(-1)\otimes L^\times)^H \arrow[r,"\gamma"] & H^2(H,A). 
		\end{tikzcd}
		\]
		As cup products and connecting homomorphisms are also functorial in $A$, if the composite of the top horizontal maps is the opposite of the cup product with $(\partial_2\partial_1)(f\mu_e)$, then $(\gamma\beta_f)(a)=-a\cup (\partial_2\partial_1)(f\mu_e)$. We may thus assume that $A=\mu_n$ with the trivial $H$-action. 
		
		In this case, under the identifications $\mu_n(-1)\otimes L^\times=\Z/n\Z\otimes L^\times = L^\times/L^{\times n}$, the map $\gamma$ takes the form
		\[\gamma\colon (L^\times/L^{\times n})^H\xrightarrow{\sim}H^1(L,\mu_n)^H\xrightarrow{\on{tg}}H^2(H,\mu_n)\]
		where the left map is the Kummer isomorphism.
		
		Let $f\in L^\times$ be such that $f\mu_{e}\in (L^\times/\mu_{e})^H$. The commutative diagram constructed in \Cref{galois-algebras} and the functoriality of the transgression map yield a commutative square
		\[
		\begin{tikzcd}
			H^1(\mu_n, \mu_n)^{\mu_e} \arrow[d,"(\rho_f)^*"] \arrow[r,"\on{tg}"] &  H^2(\mu_e,\mu_n) \arrow[d,"(\partial_1(f\mu_e))^*"]   \\
			H^1(L,\mu_n)^H \arrow[r,"\on{tg}"]     &  H^2(H,\mu_n).
		\end{tikzcd}
		\]
		If $1\in H^1(\mu_n, \mu_n)^{\mu_e}$ denotes the identity map $\mu_n\to \mu_n$, then $(\rho_f)^*(1)$ is the homomorphism corresponding to $fL^{\times n}\in L^\times/L^{\times n}$ via the Kummer isomorphism. It follows that $\gamma(f)=\on{tg}((\rho_f)^*(1))=(\partial_1(f\mu_e))^*(\on{tg}(1))$. It remains to check that $(\partial_1(f\mu_e))^*(\on{tg}(1))=-(\partial_2\partial_1)(f\mu_e)$. 
		
		Let $\theta\in H^2(\mu_e,\mu_n)$ be the class of the central extension (\ref{delta-2}). By \cite[Theorem 2.4.4]{neukirch2008cohomology}, we have $\on{tg}(1)=-\theta$. Thus, letting $\chi\coloneqq \partial_1(f\mu_e)\colon H \to \mu_e$, we have $(\partial_1(f\mu_e))^*(\on{tg}(1))=-\chi^*(\theta)$. In fact, the pullback of (\ref{delta-2}) by $\partial_1(f\mu_e)$ is the top row of (\ref{galois-algebras-diag2}). Observe that $(\partial_2\partial_1)(f\mu_e) = \partial_2(\chi)$ is also the pullback of (\ref{delta-2}) with respect to $\chi$, that is, $(\partial_2\partial_1)(f\mu_e)=\chi^*(\theta)$. Thus 
		\[(\partial_1(f\mu_e))^*(\on{tg}(1))=-\chi^*(\theta)=-(\partial_2\partial_1)(f\mu_e).\qedhere\]
	\end{proof}
	
	For the proof of \Cref{thm-negligible}, we will use the following consequence of \Cref{diagram}. For every subgroup $H'\subset H$, consider the map
	\begin{equation}\label{varphi}\varphi_{H'}\colon A^{H'}\otimes H^2(H',\Z)\xrightarrow{\cup}H^2(H',A)\xrightarrow{\on{cor}} H^2(H, A)\end{equation}
	which appears in the statement of \Cref{thm-negligible}. We define $\smash{\overbar{H}}^2(H,A)$ as the subgroup of $H^2(H,A)$ generated by the images of the $\varphi_{H'}$, where $H'$ ranges over all the subgroups of $H$.
	
	\begin{prop}\label{image-varphi}
		Let $e,n\geq 1$ be integers, let $H$ be a finite group of exponent dividing $e$, let $A$ be a finite $H$-module such that $nA=0$, and let $L/K$ be a Galois field extension with Galois group $H$, such that $K$ contains a primitive root of unity of order $ne$. For all subgroups $H'\subset H$, all $a\in A(-1)^{H'}$ and all $f\in L^\times$ such that $f\mu_e\in (L^\times/\mu_e)^{H'}$, we have 
		\[
		\gamma(N_{H/H'}(a\otimes f)) \in \on{Im}(\varphi_{H'}) \subset \smash{\overbar{H}}^2(H,A).
		\]
	\end{prop}
	
	We need the following lemma.
	
	\begin{lemma}\label{tor-trick}
		Let $H$ be a finite group, and let $0 \to X \to Y \to Z \to 0$ be an exact sequence of abelian groups considered as
		trivial $H$-modules. Then
		\[
		\on{Im}(H^1(H, Z)\xrightarrow{\partial} H^2(H, X)) \subset \on{Im}(H^2(H, \Z) \otimes X \xrightarrow{\cup} H^2(H, X)).
		\]
	\end{lemma}
	
	\begin{proof}

		We have a commutative diagram
		\[
		\begin{tikzcd}
			&& H^1(H, Z)  \arrow[d,"\partial"] \\
			0 \arrow[r] &  H^2(H, \Z) \otimes X \arrow[d]  \arrow[r,"\cup"] &  H^2(H, X) \arrow[d] \arrow[r] & \on{Tor}^{\Z}_1(H^3(H, \Z), X)  \arrow[d,hook]\arrow[r] & 0 \\
			0 \arrow[r] &  H^2(H, \Z) \otimes Y \arrow[r,"\cup"]    &  H^2(H, Y) \arrow[r] &  \on{Tor}^{\Z}_1(H^3(H, \Z), Y) \arrow[r] & 0.
		\end{tikzcd}
		\]
		The exactness of the rows is a consequence of the universal coefficient theorem: for lack of a precise reference in the cohomological setting, we sketch a proof. Pick a $\Z$-free resolution $0 \to F_1 \to F_2 \to X \to 0$, and consider the associated long exact sequence of $H$-cohomology:
		\[H^2(H,F_1)\to H^2(H,F_2)\to H^2(H,X)\to H^3(H,F_1)\to H^3(H,F_2).\]
		We have identifications $H^i(H,F_j)=H^i(H,\Z)\otimes F_j$ for all $i\geq 0$ and $j=1,2$. Now conclude by using that 
		\begin{align*}
			\on{Coker}[H^2(H,\Z)\otimes F_1\to H^2(H,\Z)\otimes F_2]&=H^2(H,\Z)\otimes X,\\
			\on{Ker}[H^3(H,\Z)\otimes F_1\to H^3(H,\Z)\otimes F_2]&=\on{Tor}^{\Z}_1(H^3(H,\Z),X).
		\end{align*}
		This completes the proof of the exactness of the rows. Since $\on{Tor}^{\Z}_2=0$, the functor $\on{Tor}^{\Z}_1$ is left exact, and so the vertical map on the right is injective. The conclusion follows from a diagram chase.
	\end{proof}

	\begin{proof}[Proof of \Cref{image-varphi}]
		We first consider the case when $H=H'$. In this case, the map $\varphi_H$ is the cup product $\cup\colon A^H\otimes H^2(H,\Z)\to H^2(H,A)$.  We apply \Cref{tor-trick} to the short exact sequence (\ref{delta-1}) and conclude that for every $f\in L^\times$ such that $f\mu_e \in (L^\times / \mu_e)^H$ we have
		\[
		(\partial_2 \partial_1)(f\mu_e) \in \on{Im}[H^2(H, \Z) \otimes \mu_n \xrightarrow{\cup} H^2(H, \mu_n)].
		\]
		It follows from \Cref{diagram} that for every $a \in A(-1)^H$ we have
		\[
		(\gamma\beta)(a \otimes f\mu_e) = -a \cup(\partial_2\partial_1)(f\mu_e) \in \on{Im}[A(-1)^H \otimes H^2(H, \Z) \otimes \mu_n \xrightarrow{\cup} H^2(H, A)].
		\]
		As $nA=0$, we have a canonical identification $A(-1) \otimes \mu_n = A$ given by evaluation. We deduce that 
		\[
		(\gamma\beta)(a \otimes f\mu_e)  = -a \cup(\partial_2\partial_1)(f\mu_e)  \in \on{Im}[A^H \otimes H^2(H, \Z) \xrightarrow{\cup} H^2(H, A)].
		\]
		As $(\gamma\beta)(a \otimes f\mu_e)=\gamma(a\otimes f)$, this completes the proof when $H'=H$.
		
		Now let $H'\subset H$ be a subgroup, let $a \in A(-1)^{H'}$, and let $f\in L^\times$ be such that $f\mu_e \in (L^\times / \mu_e)^{H'}$. By the previous case applied to $H'$, we have
		\[
		\gamma(a \otimes f) \in \on{Im}[A^{H'} \otimes H^2(H', \Z) \xrightarrow{\cup} H^2(H', A)].
		\]
		By \Cref{cor-tg} and the fact that the connecting homomorphism $\iota$ is compatible with corestrictions \cite[Proposition 1.5.2]{neukirch2008cohomology}, the map $\gamma=\on{tg}\circ \iota$ is compatible with corestrictions, and hence
		\[\gamma(N_{H/H'}(a\otimes f)) = \on{cor}^{H'}_H(\gamma(a\otimes f))\quad\text{in $H^2(H,A)$,}\]
		where on the left (resp. right) we are considering the map $\gamma$ for $H$ (resp. $H'$). We conclude that $\gamma(N_{H/H'}(a\otimes f))$ belongs to $\on{Im}(\varphi_{H'})$, as desired.
	\end{proof}
	
	\begin{rmk}
		Although we will not need this, we observe that the definition of $A(-1)$ does not depend on the choice of $n$, up to natural isomorphism: indeed, because $e(A)$ divides $n$, the natural surjective homomorphism $\pi\colon \mu_n\to \mu_{e(A)}$ induces an $H$-module isomorphism $\pi^*\colon \on{Hom}(\mu_{e(A)},A)\xrightarrow{\sim} \on{Hom}(\mu_n,A)$. Similarly, the definition of $\iota$ is independent of $n$, in the following sense: the surjection $\pi\colon \mu_n\to \mu_{e(A)}$ induces a commutative diagram
		\[
		\begin{tikzcd}
			\on{Hom}(\mu_{e(A)},A)\otimes L^\times \arrow[r,"\iota"] \arrow[d,"\wr\,\,\pi^*\otimes \on{id}"]  &  H^1(L,A) \arrow[d,equal] \\
			\on{Hom}(\mu_n,A)\otimes L^\times \arrow[r,"\iota"]   & H^1(L,A).
		\end{tikzcd}
		\]
		It follows that the definition of $\gamma$ is also independent of $n$.
	\end{rmk}

	\section{Proof of Theorem \ref{thm-negligible}}\label{section-4}
	
	Let $e,n\geq 1$ be integers, let $H$ be a finite group of exponent dividing $e$, let $A$ be a finite $H$-module such that $nA=0$, let $F$ be a field containing a primitive root of unity of order $ne$, and let $V$ be a faithful $H$-representation over $F$. In view of \Cref{connection-tg} and the definition of the map $\gamma$ of (\ref{gamma}) for the generic Galois extension $F(V)/F(V)^H$, we have
	\begin{equation}\label{connection-gamma}H^2(H,A)_{\on{neg}, F}=\on{Im}[(A(-1)\otimes F(V)^\times)^H\xrightarrow{\gamma} H^2(H,A)].
	\end{equation}
	We will prove \Cref{thm-negligible} by applying \Cref{image-varphi} to the case when $L/K$ is the generic extension $F(V)/F(V)^H$, in combination with \Cref{gener} below. 
	
	We view the faithful $H$-representation $V$ over $F$ as an affine space over $F$, and we denote by $V^{(1)}$ the set of codimension $1$ points of $V$. Because the Picard group of $V$ is trivial, we have a short exact sequence of $H$-modules
	\begin{equation}\label{div-sequence}
		\begin{tikzcd}
			1\arrow[r] &  F^\times \arrow[r] & F(V)^\times \arrow[r,"\on{div}"] & \on{Div}(V) \arrow[r] & 0,
		\end{tikzcd}
	\end{equation}
	where $\on{Div}(V)$ is the free abelian group on the set $V^{(1)}$, and where $\on{div}$ is the divisor map.
	
	\begin{lemma}\label{gener}
		Let $H$ be a finite group, let $B$ be a finite $H$-module, let $F$ be a field, and let $V$ be a faithful $H$-representation over $F$. For every $x\in V^{(1)}$, let $H_x\subset H$ be the stabilizer of $x$, and choose $f_x\in F(V)^\times$ such that $\on{div}(f_x)=x$. The group $(B\otimes F(V)^\times)^H$ is generated by $(B\otimes F^\times)^H$ and all elements of the form $N_{H/H_x}(b\otimes f_x)$, for all $x \in V^{(1)}$ and all $b\in B^{H_x}$.
	\end{lemma}
	
	\begin{proof}
		By the exactness of (\ref{div-sequence}), for every $x\in V^{(1)}$ one may find $f_x\in F(V)^\times$ such that $\on{div}(f_x)=x$, and $f_x$ is uniquely determined up to multiplication by an element of $F^\times$. As $\on{Div}(V)$ is a free abelian group, (\ref{div-sequence}) remains exact after tensorization with $B$. Taking $H$-invariants, we obtain an exact sequence
		\[
		\begin{tikzcd}
			1\arrow[r] &  (B\otimes F^\times)^H \arrow[r] & (B\otimes F(V)^\times)^H \arrow[r] & (B\otimes \on{Div}(V))^H.   
		\end{tikzcd}
		\]
		Therefore, it suffices to prove that the group $(B\otimes \on{Div}(V))^H$ is generated by all elements of the form $N_{H/H_x}(b\otimes x)$, where $b\in B^{H_x}$ and $x\in V^{(1)}$. For every $x\in V^{(1)}$, consider the $H$-submodule 
		\[M_x\coloneqq B\otimes \operatornamewithlimits{\textstyle\coprod}_{i=1}^{r} \Z\cdot h_i(x)\subset B\otimes\on{Div}(V),\]
		where $r\coloneqq [H:H_x]$, the $h_i$ form a system of representatives for $H/H_x$, and where $h_1$ is in $H_x$. Observe that $B\otimes \on{Div}(V)$ is the direct sum of the $M_x$, where $x\in V^{(1)}$ ranges over a set of representatives of the $H$-orbits. In order to conclude, it now suffices to observe that an element $m=\sum_{i=1}^r b_i\otimes h_i(x)\in M_x$ is $H$-invariant if and only if $b_1$ belongs to $B^{H_x}$ and $b_i=h_i(b_1)$ for all $i$, that is, letting $b\coloneqq b_1\in B^{H_x}$, if and only if $m=\sum_{i=1}^{k} h_i(b)\otimes h_i(x)=N_{H/H_x}(b\otimes x)$.
	\end{proof}

	\begin{proof}[Proof of Theorem \ref{thm-negligible}]
		We let $e\coloneqq e(H)$ and $n\coloneqq e(A)$. Recall that we defined the subgroup $\smash{\overbar{H}}^2(H,A)\subset H^2(H,A)$ as the image of the maps $\varphi_{H'}$ of (\ref{varphi}), where $H'$ ranges over all subgroups of $H$. With this notation, \Cref{thm-negligible} is equivalent to the assertion $H^2(H,A)_{\on{neg},F}=\smash{\overbar{H}}^2(H,A)$. 
		
		We first show that $\smash{\overbar{H}}^2(H,A)\subset H^2(H,A)_{\on{neg},F}$.	Let $H'\subset H$ be a subgroup. By \Cref{props}(3), to show that $H^2(H,A)_{\on{neg},F}$ contains the image of $\varphi_{H'}$, it is enough to prove that $H^2(H',A)_{\on{neg},F}$ contains the image of the cup product map $A^{H'}\otimes H^2(H',\Z)\to H^2(H',A)$. We may thus assume that $H'=H$.
		
		Let $a\in A^H$, and let $c\in H^2(H,\Z)$. We must show that $a\cup c\in H^2(H,A)$ is negligible over $F$. The connecting map $\partial\colon H^1(H,\Q/\Z)\to H^2(H,\Z)$ associated to the sequence $0\to \Z\to \Q\to \Q/\Z\to 0$ is an isomorphism. Let $\chi\in H^1(H,\Q/\Z)$ be such that $\partial(\chi)=c$, let $\smash{\overbar{H}}\coloneqq \on{Im}(\chi)\subset \Z/e\Z$, and let $\overline{\chi}\colon \smash{\overbar{H}}\to \Q/\Z$ be the character induced by $\chi$. Then $\chi$ is the inflation of $\overline{\chi}\in H^1(\smash{\overbar{H}},\Q/\Z)$, and so, by the compatibility of connecting maps with inflation maps \cite[Proposition 1.5.2]{neukirch2008cohomology}, $c$ is the inflation of $\overline{c}\coloneqq \partial(\overline{\chi})\in H^2(\smash{\overbar{H}},\Z)$, where now $\partial\colon H^1(\smash{\overbar{H}},\Q/\Z)\to H^2(\smash{\overbar{H}},\Z)$ is the connecting map in $\smash{\overbar{H}}$-cohomology. Let $\psi_a\colon \Z/n\Z\to A$ be the homomorphism of $H$-modules sending $1$ to $a$, where we view $\Z/n\Z$ as an $H$-module with trivial action. We have a commutative diagram
		\[
		\begin{tikzcd}
			\Z/n\Z \otimes H^2(\smash{\overbar{H}},\Z) \arrow[d,"\cup"] \arrow[r,"\on{id}\otimes\on{inf}"] &   \Z/n\Z\otimes H^2(H,\Z) \arrow[r,"(\psi_a)_*\otimes\on{id}"]\arrow[d,"\cup"]  & A^H\otimes H^2(H,\Z) \arrow[d,"\cup"] \\
			H^2(\smash{\overbar{H}},\Z/n\Z)  \arrow[r,"\on{inf}"] & H^2(H,\Z/n\Z) \arrow[r,"(\psi_a)_*"] & H^2(H,A). 
		\end{tikzcd}
		\]
		Now \Cref{props}(1) and (2) imply that if $1\cup\overline{c}\in H^2(\smash{\overbar{H}},\Z/n\Z)$ is negligible over $F$, then so is $a\cup c$. This reduces us to the following situation: $H=\Z/m\Z$ for some integer $m$ dividing $e$, $A=\Z/n\Z$, and $H$ acts trivially on $A$. Since $F$ contains a primitive root of order $ne$ and $m$ divides $e$, by \Cref{kummer-theory} every class in $H^2(\Z/m\Z,\Z/n\Z)$ is negligible over $F$, as desired. Thus $\smash{\overbar{H}}^2(H,A)\subset H^2(H,A)_{\on{neg},F}$.
		
		We now show that $H^2(H,A)_{\on{neg},F}\subset \smash{\overbar{H}}^2(H,A)$. By (\ref{connection-gamma}), it is enough to show that the map $\gamma$ of (\ref{gamma}) takes $(A(-1)\otimes F(V)^\times)^H$ to $\smash{\overbar{H}}^2(H,A)$. For this, we will combine \Cref{gener} and \Cref{image-varphi} for $n=e(A)$ and $e=e(H)$. Let $\smash{\overbar{F}}$ be an algebraic closure of $F$. As the map $\gamma$ is functorial in $F$, we have a commutative diagram
		\[
		\begin{tikzcd}
			(A(-1)\otimes F^\times)^H \arrow[r,hook]\arrow[d] & (A(-1)\otimes F(V)^\times)^H \arrow[r,"\gamma"]\arrow[d]  & H^2(H,A) \arrow[d,equal]  \\
			(A(-1)\otimes \overbar{F}\vphantom{F}^\times)^H \arrow[r,hook] & (A(-1)\otimes \overbar{F}(V)^\times)^H \arrow[r,"\gamma"] & H^2(H,A).
		\end{tikzcd}
		\]
		As $\smash{\overbar{F}}\vphantom{F}^\times$ is divisible and $A(-1)$ is finite, we have $A(-1)\otimes \smash{\overbar{F}}\vphantom{F}^\times=0$. It follows that $\gamma$ maps $(A(-1)\otimes F^\times)^H$ to zero. 
		
		Now let $x\in V^{(1)}$, let $H_x\subset H$ be the stabilizer of $x$, choose $f\in F(V)^\times$ such that $\on{div}(f)=x$, and let $a\in A(-1)^{H_x}$. In order to apply \Cref{image-varphi}, we must first check that $f\mu_e\in (F(V)^\times/\mu_e)^{H_x}$. For this, define $\chi\in H^1(H_x,F^\times)$ as the image of $x\in \on{Div}(X)$ under the connecting map for (\ref{div-sequence}), viewed as a sequence of $H_x$-modules. Since $H$ has exponent $e$, the image of $\chi\colon H_x\to F^\times$ is contained in $\mu_e$. By definition of the connecting map, for all $h\in H_x$ we have $(h-1)f=\chi(h)\in \mu_e$. Thus $f\mu_e\in (F(V)^\times/\mu_e)^{H_x}$, as desired. By \Cref{image-varphi}, the element $\gamma(N_{H/H_x}(a\otimes f))$ belongs to $\on{Im}(\varphi_{H_x})$. By \Cref{gener} applied to $B=A(-1)$, we conclude that the image of $\gamma$ is contained in $\smash{\overbar{H}}^2(H,A)$. This completes the proof.
	\end{proof}
	
	\begin{cor}\label{maincor}
		Under the assumptions of \Cref{thm-negligible}, let $H_a$ be the centralizer in $H$ of an element $a\in A$. Then
		$H^2(H, A)_{\on{neg},F}$ is generated by $\on{cor}_H^{H_a}(a\cup \chi)$ over all $a\in A$ and $\chi\in H^2(H_a,\Z)$.
	\end{cor}
	
	\begin{proof}
		Let $M\subset H^2(H,A)$ be the subgroup generated by $\on{cor}_H^{H_a}(a\cup \chi)$ over all $a\in A$ and $\chi\in H^2(H_a,\Z)$. By \Cref{thm-negligible}, it suffices to show that $M=\smash{\overbar{H}}^2(H,A)$.
		
		We have $M\subset \smash{\overbar{H}}^2(H,A)$ by definition. Conversely, let $H'\subset H$ be a subgroup, let $a\in A^{H'}$, and let $\chi \in H^2(H',\Z)$. Then $H'\subset H_a$, and hence by the projection formula
		\[
		\on{cor}_H^{H'}(a\cup \chi)=\on{cor}_H^{H_a}\on{cor}_{H_a}^{H'}(a\cup \chi)=\on{cor}_H^{H_a}(a\cup \on{cor}_{H_a}^{H'}(\chi)).
		\]
		Thus $\smash{\overbar{H}}^2(H,A)\subset M$, and hence $\smash{\overbar{H}}^2(H,A)=M$.
	\end{proof}
	
	\begin{rmk}
		Let $F$ be a field. For all integers $m,n\geq 1$ such that $F$ contains a primitive root of unity of order $mn$, we call \emph{Kummer} the class in $H^2(\Z/m\Z,\Z/n\Z)$ of the central extension
		\[
		\begin{tikzcd}
			0\arrow[r] & \Z/n\Z \arrow[r] & \Z/mn\Z \arrow[r,"\pi"] & \Z/m\Z \arrow[r] & 0,
		\end{tikzcd}
		\]
		where $\pi$ is the reduction map. On $H^2(-,-)$, where the first entry is a finite group and the second entry is a finite module under the group, we can consider the following operations:
		\begin{enumerate}
			\item pushforward with respect to the module,
			\item pullback with respect to the group,
			\item corestriction with respect to a subgroup.
		\end{enumerate}
		Let $H$ be a finite group, and let $A$ be a finite $H$-module. By \Cref{kummer-theory} and \Cref{props}, the classes in $H^2(H,A)$ that can be obtained from Kummer classes by applying operations (1)-(3) a finite number of times, in any order, are negligible over $F$. It follows from the proof of \Cref{thm-negligible} that, if $F$ contains a primitive root of unity of order $e(A)e(H)$, then the classes in $H^2(H,A)$ obtained from Kummer classes using operations (1)-(3) form a generating set for $H^2(H,A)_{\on{neg}, F}$.
	\end{rmk}

	\section{Proof of Theorem \ref{thm-galois-rep}}\label{section-5}

	In view of \Cref{connection}, the following theorem gives a negative answer to \Cref{question} and disproves \Cref{florence-conj}.
	
	\begin{thm}\label{thm-gl3}
		For all $n\geq 3$, all primes $p\neq 2$, and all fields $F$ of characteristic different from $p$, the cohomology classes of the extensions
		\begin{equation}\label{gl-eq} 
			\begin{tikzcd}
				0 \arrow[r] & \mathfrak{gl}_n(\F_p) \arrow[r] & \on{GL}_n(\Z/p^2\Z) \arrow[r] & \on{GL}_n(\F_p) \arrow[r] & 1
			\end{tikzcd}
		\end{equation}
		and
		\begin{equation}\label{b-eq}
			\begin{tikzcd}
				0 \arrow[r] & \mathfrak{b}_n(\F_p) \arrow[r] & B_n(\Z/p^2\Z) \arrow[r] & B_n(\F_p) \arrow[r] & 1
		\end{tikzcd}\end{equation}
		are not negligible over $F$.
	\end{thm}

	In the rest of this section, we identify $\mathfrak{gl}_n(\F_p)$ with the $\F_p$-vector space $M_n(\F_p)$ of $n\times n$ matrices. Under this identification, the conjugation action of $\on{GL}_n(\Z/p^2\Z)$ on $\mathfrak{gl}_n(\F_p)$ induces the conjugation action of $\on{GL}_n(\F_p)$ on $M_n(\F_p)$, and the inclusion $\mathfrak{gl}_n(\F_p)\hookrightarrow\on{GL}_n(\Z/p^2\Z)$ of (\ref{gl-eq}) is given by $M\mapsto I+pM$. 
	
	\begin{lemma}\label{matrix-computation}
		Let $n\geq 1$ be an integer, let $p>3$ be a prime, and let $\sigma\in \on{GL}_n(\F_p)$ be an element of order $p$ and minimal polynomial $(t-1)^2$.
		
		(1) For every $M\in \mathfrak{gl}_n(\F_p)$, we have $N_\sigma(M)\coloneqq (1+\sigma+\dots+\sigma^{p-1})M=0$.
		
		(2) The element $\sigma$ does not lift to an element of order $p$ in $\on{GL}_n(\Z/p^2\Z)$.
	\end{lemma}
	
	\begin{proof}
		(1) Up to conjugation, we may assume that $\sigma=I+N$, where $N$ is a nilpotent matrix in Jordan normal form with Jordan blocks of size $\leq 2$. Observe that $\sigma^i=(I+N)^i=I+iN$ for all $i\geq 0$. Therefore, for all $M\in \mathfrak{gl}_n(\F_p)$, each entry of $\sigma^i\cdot M = (I+iN)M(I-iN)$ is a polynomial function of $i$ with $\F_p$ coefficients and degree $\leq 2$. Since $p>3$, we have $1^j+\dots+p^j=0$ in $\F_p$ for $0\leq j\leq 2$. Thus
		$N_\sigma(M)=0$ in $\mathfrak{gl}_n(\F_p)$.
		
		(2) Up to conjugation, we may assume that $\sigma=I+N$ as in the proof of (1). Let $\tilde{\sigma}=I+\tilde{N}\in \on{GL}_n(\Z/p^2\Z)$, where $\tilde{N}$ is the only nilpotent matrix in Jordan form which reduces to $N$ modulo $p$. Then $\tilde{\sigma}$ is a lift of $\sigma$ and $\tilde{\sigma}^p=I+p\tilde{N}$, so that in particular $\tilde{\sigma}^p\neq 1$. If $\tau\in \on{GL}_n(\Z/p^2\Z)$ is another lift of $\sigma$, then $\tau=\tilde{a}\tilde{\sigma}$, where $\tilde{a}=I+pa$ for some $a\in \mathfrak{gl}_n(\F_p)$. 
		It follows from (1) that 
		\[\tau^p=(\tilde{a}\tilde{\sigma})^p=\tilde{\sigma}^p+pN_\sigma(a)=\tilde{\sigma}^p\neq 1.\qedhere\]
	\end{proof}

	\begin{lemma}\label{reduce-n-3}
		Let $F$ be a field, let $\tilde{V}$ and $\tilde{W}$ be free $\Z/p^2\Z$-modules of finite rank, and let $V$ and $W$ be their reductions modulo $p$. If the class of
		\[
		\begin{tikzcd}
			0 \arrow[r] &  \mathfrak{gl}(V\oplus W) \arrow[r] & \on{GL}(\tilde{V}\oplus \tilde{W}) \arrow[r] & \on{GL}(V\oplus W)\arrow[r] &  1
		\end{tikzcd}
		\]
		is negligible over $F$, then so is the class of
		\[
		\begin{tikzcd}
			0 \arrow[r] &  \mathfrak{gl}(V) \arrow[r] & \on{GL}(\tilde{V}) \arrow[r] & \on{GL}(V)\arrow[r] &  1.
		\end{tikzcd}
		\]
	\end{lemma}
	
	\begin{proof}
		This follows from \cite[Lemma 3.4]{declercq2017lifting}. We include an alternative argument. Let $E\subset \on{GL}(\tilde{V}\oplus\tilde{W})$ be the inverse image of $\on{GL}(V)\subset \on{GL}(V\oplus W)$. The restriction to $\tilde{V}$ of any element of $E$ belongs to $\on{GL}(\tilde{V})$. It follows that restriction to $\tilde{V}$ defines a group homomorphism $\pi\colon E\to \on{GL}(\tilde{V})$. We obtain the following commutative diagram with exact rows:
		\[
		\begin{tikzcd}
			0 \arrow[r] &  \mathfrak{gl}(V\oplus W) \arrow[r] & \on{GL}(\tilde{V}\oplus \tilde{W}) \arrow[r] & \on{GL}(V\oplus W)\arrow[r] &  1 \\
			0\arrow[r] &  \mathfrak{gl}(V\oplus W) \arrow[r] \arrow[d] \arrow[u,equal] & E \arrow[r] \arrow[d,"\pi"] \arrow[u,hook] & \on{GL}(V) \arrow[r]\arrow[d,equal] \arrow[u,hook] & 1 \\
			0\arrow[r] &  \mathfrak{gl}(V) \arrow[r] & \on{GL}(\tilde{V}) \arrow[r] & \on{GL}(V) \arrow[r] & 1.
		\end{tikzcd}
		\]
		As the top row of this diagram is negligible over $F$ by assumption, by \Cref{props}(2) so is the middle row, and hence by \Cref{props}(1) so is the bottom row.
	\end{proof}
	
	\begin{proof}[Proof of \Cref{thm-gl3}]
		Suppose that the class of (\ref{b-eq}) is negligible over $F$, for some $n\geq 3$. Then, by \Cref{props}(1), the class of the restriction of (\ref{b-eq}) to the group $H^2(B_n(\F_p),\mathfrak{gl}_n(\F_p))$ is also negligible over $F$. This class is the restriction of the class of (\ref{gl-eq}) to $B_n(\F_p)$. Since $p$ does not divide $[\on{GL}_n(\F_p): B_n(\F_p)]$, by \Cref{props}(4) we conclude that the class of (\ref{gl-eq}) is also negligible over $F$. This reduces the proof of \Cref{thm-gl3} to proving that (\ref{gl-eq}) is not negligible over $F$, for all $n\geq 3$. By \Cref{reduce-n-3}, we are further reduced to the case $n=3$.
		
		We let $A\coloneqq \mathfrak{gl}_3(\F_p)$, $B\coloneqq B_3(\F_p)$, $U\coloneqq U_3(\F_p)$, $T\coloneqq T_3(\F_p)$, where $T_3\subset \on{GL}_3$ is the diagonal torus, and we let $\alpha\in H^2(U,A)$ be the class of the restriction of (\ref{gl-eq}) to $U$. Concretely, $\alpha$ represents the extension
		\begin{equation}\label{restriction-alpha}
			\begin{tikzcd}
				1\arrow[r] & A \arrow[r] & \tilde{U} \arrow[r] & U \arrow[r] & 1,
			\end{tikzcd}
		\end{equation}
		where $\tilde{U}\subset \on{GL}_3(\Z/p^2\Z)$ is the inverse image of $U$.
		Note that $\alpha$ belongs to the image of the restriction $H^2(B,A)\to H^2(U,A)$, and so it is $T$-invariant; see \cite[Proposition 3 p. 99]{cassels1967algebraic}. In order to complete the proof, by \Cref{props}(2) it is enough to prove that $\alpha$ is not negligible over $F$. Since $\on{char}(F)\neq p$, by \Cref{props}(5) we may assume that $F$ contains a primitive root of unity of order $p^2$. We have $e(A)=p$ and, as $p$ is odd, $e(U)=p$. We are now in the setting of \Cref{thm-negligible}, and hence $H^2(U,A)_{\on{neg},F}=\smash{\overbar{H}}^2(U,A)$. It remains to show that $\alpha$ does not belong to $\smash{\overbar{H}}^2(U,A)$.
		
		We introduce some notation that will be used throughout the remainder of the proof. For each $i,j\in\{1,2,3\}$, let $E_{ij}\in A$ be the matrix with a $1$ in the $(i,j)$-th position and $0$ everywhere else, let $\tau_{ij}\colon T\to \F_p^\times$ be the multiplicative character  which sends $(t_1,t_2,t_3)$ to $t_i/t_j$, and let $\F_p(\tau_{ij})$ be the $1$-dimensional representation of $T$ given by multiplication by $\tau_{ij}$. As a $T$-representation, $A$ decomposes as the direct sum of the one-dimensional subrepresentations spanned by the $E_{ij}$, and these are isomorphic to $\F_p(\tau_{ij})$. For every $1\leq i<j\leq 3$, let $\sigma_{ij}\in U$ be the matrix with $1$ on the diagonal and on the $(i,j)$-th entry, and $0$ everywhere else. We also define $N\coloneqq \ang{\sigma_{12},\sigma_{13}}\subset U$ and $Z\coloneqq \ang{\sigma_{13}}\subset U$. Observe that $N$ is normal in $U$, that $Z$ is the center of $U$, and that $N$ and $Z$ are normalized by $T$.
		
		To prove that $\alpha$ does not belong to $\smash{\overbar{H}}^2(U,A)$, we will first show that $\alpha$ restricts to a non-zero element of $H^2(N,A)^T$ (\Cref{claim-1}), and then prove that $\on{res}^U_N(\smash{\overbar{H}}^2(U,A))^T=0$ (\Cref{claim-4}).
		
		\begin{claim}\label{claim-1}
			The class $\alpha$ restricts to a non-zero element of $H^2(N,A)^T$.
		\end{claim}
		
		\begin{proof}
			Since $T$ normalizes $U$ and $N$, the inclusion $N\hookrightarrow U$ is $T$-equivariant. As $\alpha$ is $T$-invariant, by \cite[Proposition 3 p. 99]{cassels1967algebraic} the class $\on{res}^U_N(\alpha)$ is also $T$-invariant. It remains to show that $\on{res}^U_N(\alpha)$ is non-trivial. 
			
			Suppose first that $p>3$. In this case, by \Cref{matrix-computation} the element $\sigma_{13}\in U$ does not lift to an element of order $p$ in $\on{GL}_3(\Z/p^2\Z)$. Thus $\on{res}^U_Z(\alpha)$ is non-trivial extension, and hence in particular $\on{res}^U_N(\alpha)$ is non-trivial.
			
			Suppose now that $p=3$. If $\on{res}^U_N(\alpha)=0$ in $H^2(N,A)$, then $N$ should lift to an isomorphic subgroup of $\on{GL}_3(\Z/9\Z)$, that is, there should exist two elements $\sigma,\tau\in \on{GL}_3(\Z/9)$ such that $\sigma^3=\tau^3=[\sigma,\tau]=1$, $\sigma$ reduces to $\sigma_{12}$ modulo $3$, and $\tau$ reduces to $\sigma_{13}$ modulo $3$. We now prove that such $\sigma$ and $\tau$ do not exist. Indeed, let $\tilde{\sigma}_{12} = I+E_{12}$ and $\tilde{\sigma}_{13} = I+E_{13}$ in $GL_3(\Z/9\Z)$. We have
			$(\tilde{\sigma}_{12})^3 =I+3E_{12}$, $(\tilde{\sigma}_{13})^3 = I+3E_{13}$ and $[\tilde{\sigma}_{12},\tilde{\sigma}_{13}]=1$.
			All the possible $\sigma$ (resp. $\tau$) are of the form $\tilde{a}\tilde{\sigma}_{12}$ (resp. $\tilde{b}\tilde{\sigma}_{13}$), for some $a,b\in A$, where $\tilde{a}=I+3a$ and $\tilde{b}=I+3b$. Define $\sigma\coloneqq \tilde{a}\tilde{\sigma}_{12}$ and $\tau\coloneqq \tilde{b}\tilde{\sigma}_{13}$. The conditions $\sigma^3=\tau^3=[\sigma,\tau]=1$ are equivalent to $N_{\sigma_{12}}(a) = -E_{12}$, $N_{\sigma_{13}}(b) = -E_{13}$
			and $(\sigma_{13}-1)a = (\sigma_{12} - 1)b$. Here we write the $U$-action on $A$ in additive notation, and we let $N_{\sigma_{ij}}\coloneqq 1+\sigma_{ij}+\sigma_{ij}^2$. Letting $a=(a_{ij})$, a matrix computation shows 
			\[N_{\sigma_{12}}(a)=
			\begin{pmatrix}
				0 & a_{21} & 0 \\
				0 & 0 & 0 \\
				0 & 0 & 0
			\end{pmatrix},\qquad
			-(\sigma_{13}-1)a+(\sigma_{12}-1)b=\begin{pmatrix}
				* & * & * \\
				* & * & a_{21} \\
				* & * & *
			\end{pmatrix}.
			\]
			Thus $-1=a_{21}=0$, a contradiction.
			It follows that no such $\sigma$ and $\tau$ exist, and hence $\on{res}^U_N(\alpha)$ is non-trivial when $p=3$ as well.
		\end{proof}
		
		\begin{claim}\label{claim-2}
			Let $H_1\subset H_2\subset U$ be subgroups of $U$, such that $|H_i|=p^i$ for $i=1,2$. Then $\on{Im}(\varphi_{H_1})\subset \on{Im}(\varphi_{H_2})$.
		\end{claim}
		
		\begin{proof}
			For $i=1,2$, we write $\partial\colon H^1(H_i,\Q/\Z)\to H^2(H_i,\Z)$ for the connecting map associated to the short exact sequence $0\to\Z\to\Q\to\Q/\Z\to 0$.
			
			Let $\chi_1\in H^1(H_1,\Q/\Z)$ be a generator. As $e(U)=p$, we have $H_2\cong (\Z/p\Z)^2$, and hence $H_1$ is a direct factor of $H_2$. It follows that there exists $\chi_2\in H^1(H_2,\Q/\Z)$ such that $\chi_1$ is the restriction of $\chi_2$. Thus $\partial(\chi_1)$ is the restriction of $\partial(\chi_2)$, and so by the projection formula \cite[Proposition 1.5.3 (iv)]{neukirch2008cohomology}, the square
			\[
			\begin{tikzcd}
				A^{H_1} \arrow[d,"N_{H_2/H_1}"] \arrow[r,"\cup\partial(\chi_1)"] & H^2(H_1,A) \arrow[d,"\on{cor}^{H_1}_{H_2}"]  \\
				A^{H_2} \arrow[r,"\cup \partial(\chi_2)"] & H^2(H_2,A)
			\end{tikzcd}
			\]
			commutes. We deduce that $\varphi_{H_1}(a\otimes\partial(\chi_1))=\varphi_{H_2}(N_{H_2/H_1}(a)\otimes\partial(\chi_2))$ for all $a\in A^{H_1}$, from which the claim follows.
		\end{proof}
		
		Consider the map
		\[\psi_Z\colon A^Z\otimes H^2(Z,\Z)\xrightarrow{\cup} H^2(Z,A)\xrightarrow{\on{cor}}H^2(N,A).\]
		
		\begin{claim}\label{claim-3}
			The subgroup $\on{res}^U_N(\smash{\overbar{H}}^2(U,A))\subset H^2(N,A)$ is contained in the subgroup generated by the images of $\on{res}^U_N\circ\varphi_U$, $\on{res}^U_N\circ\varphi_N$, and $\psi_Z$.
		\end{claim}
		
		\begin{proof}
			By definition, $\smash{\overbar{H}}^2(N,A)$ is generated by the images of all the $\varphi_{H}$, where $H$ is a subgroup of $U$.
			
			Let $H\subset U$ be a subgroup of order $p^2$ such that $H\neq N$. Then $U=NH$ and $N\cap H=Z$, and hence by \cite[Proposition 1.5.3(iii) and Corollary 1.5.8]{neukirch2008cohomology} we have a commutative diagram
			\[
			\begin{tikzcd}
				A^{H}\otimes H^2(H,\Z) \arrow[r,"\cup"] \arrow[d,"\on{res}\otimes\on{res}"]  &  H^2(H,A) \arrow[r,"\on{cor}"] \arrow[d,"\on{res}"] & H^2(U,A) \arrow[d,"\on{res}"] \\
				A^Z\otimes H^2(Z,\Z)  \arrow[r,"\cup"] & H^2(Z,A) \arrow[r,"\on{cor}"] & H^2(N,A).
			\end{tikzcd}
			\]
			The composite of the maps of the top (resp. bottom) row is $\varphi_{H}$ (resp. $\psi_Z$). It follows that $\on{Im}(\on{res}^U_N\circ \varphi_{H})\subset \on{Im}(\psi_Z)$.
			
			If $H\subset U$ has order $p$, then it is contained in a subgroup $\tilde{H}\subset U$ of order $p^2$, and hence by \Cref{claim-2} we have $\on{Im}(\varphi_{H})\subset \on{Im}(\varphi_{\tilde{H}})$. If $\tilde{H}=N$, then $\on{Im}(\varphi_{H})\subset \on{Im}(\varphi_{N})$, which implies $\on{Im}(\on{res}^U_N\circ \varphi_{H})\subset \on{Im}(\on{res}^U_N\circ \varphi_{N})$. On the other hand, if $\tilde{H}\neq N$, then $\on{Im}(\on{res}^U_N\circ \varphi_{H})\subset \on{Im}(\psi_Z)$ by the case of subgroups of order $p^2$ considered above.
		\end{proof}
		
		\begin{claim}\label{claim-4}
			The subgroup $\on{res}^U_N(\smash{\overbar{H}}^2(U,A))^T\subset H^2(N,A)$ is trivial.
		\end{claim}
		
		\begin{proof}
			As $p$ does not divide the order of $T$, taking $T$-invariants is exact on $T$-representations with $\F_p$ coefficients. Thus, in view of \Cref{claim-3}, in order to prove that $\on{res}^U_N(\smash{\overbar{H}}^2(U,A))^T=0$ it suffices to show that $\on{Im}(\on{res}^U_N\circ\varphi_U)^T$, $\on{Im}(\on{res}^U_N\circ\varphi_N)^T$ and $\on{Im}(\psi_Z)^T$ are trivial. 
			
			We first show that $\on{Im}(\on{res}^U_N\circ\varphi_U)^T=0$. The map $\varphi_U$ is given by cup product, and hence it is $T$-equivariant. Thus $\on{res}^U_N\circ\varphi_U$ is $T$-equivariant. A matrix computation shows that $A^U=\langle I, E_{13}\rangle$. Moreover, letting $\partial\colon H^1(U,\Q/\Z)\xrightarrow{\sim} H^2(U,\Z)$ be the connecting homomorphism for $0\to \Z\to \Q\to \Q/\Z\to 0$, we have \[H^2(U,\Z)=\F_p\cdot\partial(\chi_{12})\oplus \F_p\cdot\partial(\chi_{23}),\]
			where $\chi_{ij}\colon U\to \F_p$ is the projection to the $(i,j)$-th coordinate. The $T$-action on $H^1(U,\Q/\Z)=\on{Hom}(U,\Q/\Z)$ is induced by the $T$-action on $U$ and the trivial action on $\Q/\Z$: in formulas, for every character $\chi\colon U\to \F_p$, every $t\in T$, and every $u\in U$, we have $(t\cdot \chi)(u)=\chi(t^{-1}ut)$. By \cite[Proposition 1.5.2]{neukirch2008cohomology}, the connecting homomorphism $\partial$ is $T$-equivariant, and hence $H^2(U,\Z)\cong \F_p(\tau_{21})\oplus\F_p(\tau_{32})$ as a $T$-representation. Therefore, as $T$-representations,
			\begin{align*}
				A^U\otimes H^2(U,\Z)&\cong (\F_p\oplus \F_p(\tau_{13}))\otimes (\F_p(\tau_{21})\oplus \F_p(\tau_{32})) \\
				&\cong \F_p(\tau_{21})\oplus \F_p(\tau_{32})\oplus\F_p(\tau_{23})\oplus \F_p(\tau_{12}).
			\end{align*}
			In particular $(A^U\otimes H^2(U,\Z))^T=0$, and hence $\on{Im}(\on{res}^U_N\circ\varphi_U)^T=0$.
			
			Next, we show that $\on{Im}(\on{res}^U_N\circ\varphi_N)^T=0$. Let $S\subset U$ be the subgroup generated by $\sigma_{23}$. Then $N\cap S=\{1\}$ and $U=NS$, so that by \cite[Corollary 1.5.7]{neukirch2008cohomology} the map $\on{res}^U_N\circ\on{cor}^N_U\colon H^2(N,A)\to H^2(N,A)$ is given by multiplication by the $S$-norm $N_S\coloneqq \sum_{s\in S}s$. Since $T$ normalizes $S$, we deduce that $\on{res}^U_N\circ\on{cor}^N_U$ is $T$-equivariant, and hence that $\on{res}^U_N\circ \varphi_N\colon A^N\otimes H^2(N,\Z)\to H^2(N,A)$ is $T$-equivariant. 
			
			We have $A^N=\ang{I,E_{12},E_{13}}$ and $H^2(N,\Z)=\F_p\cdot\partial(\chi_{12})\oplus \F_p\cdot\partial(\chi_{13})$, where now $\partial$ denotes the connecting homomorphism $\partial\colon H^1(N,\Q/\Z)\xrightarrow{\sim} H^2(N,\Z)$. As $T$ normalizes $N$, the map $\partial$ is $T$-equivariant. It follows that, as $T$-representations,
			\[A^N\otimes H^2(N,\Z)\cong (\F_p\oplus \F_p(\tau_{12})\oplus \F_p(\tau_{13}))\otimes (\F_p(\tau_{21})\oplus \F_p(\tau_{31})).\]
			In particular, we have
			\[(A^N\otimes H^2(N,\Z))^T=\ang{E_{12}\otimes\partial(\chi_{12}),E_{13}\otimes\partial(\chi_{13})}.\]
			We have $s\cdot E_{13}=E_{13}$ for every $s\in S$. Moreover, a matrix computation shows that $\sigma_{23}\cdot\chi_{12}=\chi_{12}$ and $\sigma_{23}\cdot\chi_{13}=\chi_{13}+\chi_{12}$ in $H^1(N,\Q/\Z)$. As $S$ normalizes $N$, the map $\partial$ is $S$-equivariant. Therefore, using that $H^2(N,\Z)$ is $p$-torsion and $p$ is odd, we get
			\[N_S(\partial(\chi_{12}))=0,\quad N_S(\partial(\chi_{13}))=p\partial(\chi_{13})+\frac{p(p-1)}{2}\partial(\chi_{12})=0\quad\text{in $H^2(N,\Z)$.}\]
			All in all, we get
			\[N_S(E_{13}\cup\partial(\chi_{13}))=E_{13}\cup N_S(\partial(\chi_{13}))=0\quad \text{in $H^2(N,A)$.}\] 
			Since $\chi_{12}$ extends to a character of $U$, by the projection formula $\varphi_N(E_{12}\otimes\partial(\chi_{12}))$ belongs to the image of $\varphi_U$, and hence $\on{res}^U_N(\varphi_U(E_{12}\otimes\partial(\chi_{12})))=0$ by the previous case. Thus $\on{Im}(\on{res}^U_N\circ\varphi_N)^T=0$.
			
			Finally, we show that $\on{Im}(\psi_Z)^T=0$. Since $T$ normalizes $Z$ and $N$, by \cite[Proposition 1.5.4]{neukirch2008cohomology} the corestriction $H^2(Z,A)\to H^2(N,A)$ is $T$-equivariant, and hence so is $\psi_Z$. Observe that $A^Z$ is contained in the subspace of upper-triangular matrices. We have $H^2(Z,\Z)=\F_p\cdot\partial(\chi_{13})$, where now $\partial$ denotes the connecting homomorphism $\partial\colon H^1(Z,\Q/\Z)\xrightarrow{\sim}H^2(Z,\Z)$. As $T$ normalizes $Z$, the map $\partial$ is $T$-equivariant, and hence $H^2(Z,\Z)\cong \F_p(\tau_{13})$. It follows that 
			\[(A^Z\otimes H^2(Z,\Z))^T=\ang{E_{13}\otimes \partial(\chi_{13})}.\] 
			(In fact, when $p>3$, we even have $(A\otimes H^2(Z,\Z))^T=\ang{E_{13}\otimes \partial(\chi_{13})}$. When $p=3$, we have $(A\otimes H^2(Z,\Z))^T=\ang{E_{13}\otimes \partial(\chi_{13}),E_{31}\otimes\partial(\chi_{13})}$, but $E_{31}$ does not belong to $A^Z$.) 
			Observe that $\chi_{13}\colon Z\to \Q/\Z$ is the restriction of $\chi_{13}\colon N\to \Q/\Z$. It follows that $E_{13}\cup \partial(\chi_{13})\in H^2(Z,A)$ is the restriction of $E_{13}\cup \partial(\chi_{13})$, viewed as an element of $H^2(N,A)$. Therefore
			\begin{align*}
				\psi_Z(E_{13}\otimes \partial(\chi_{13}))&=\on{cor}^Z_N(E_{13}\cup\partial(\chi_{13})) \\
				&=\on{cor}^Z_N(\on{res}^N_Z(E_{13}\cup\partial(\chi_{13}))) \\
				&=p(E_{13}\cup\partial(\chi_{13}))\\
				&=0.
			\end{align*}
			This proves that $\on{Im}(\psi_Z)^T=0$ as well.
		\end{proof}
		
		\Cref{claim-1} and \Cref{claim-4} imply that $\alpha$ does not belong to $\smash{\overbar{H}}^2(U,A)$. This completes the proof.
	\end{proof}

	\begin{rmk}
		(1) For $p>3$, one can simplify the proof of \Cref{thm-gl3}. Indeed, let $Z$ be the center of $U=U_3(\F_p)$, let $A=\mathfrak{gl}_3(\F_p)$, let $\alpha\in H^2(U,A)$ be the class of (\ref{gl-eq}) for $n=3$ viewed as an $U$-module. Since $p>3$, the restriction of $\alpha$ to $H^2(Z,A)$ is non-zero; see the proof of \Cref{claim-1}. Moreover, again because $p>3$, by \Cref{matrix-computation} we have $N_Z(A)=0$, and using this and the double coset formula \cite[Proposition 1.5.6]{neukirch2008cohomology} one can show that all classes in $\smash{\overbar{H}}^2(U,A)$ restrict to zero in in $H^2(Z,A)$. This proves that $\alpha$ is not negligible over $F$ for $p>3$. This argument does not work for $p=3$ because $\on{res}^U_Z(\alpha)=0$ in $H^2(Z,A)$ in this case.
		
		(2) If the restriction of the class $\alpha$ to $H^2(N,A)$ were not negligible over $F$, one could use this to simplify the proof of \Cref{thm-gl3}. However, this is not the case: one can show that $\alpha$ restricts to a negligible class in $H^2(N,A)$.
	\end{rmk}

	\begin{proof}[Proof of \Cref{thm-galois-rep}]
		Let $U\coloneqq U_3(\F_p)\subset \on{GL}_3(\F_p)$, and consider the following commutative diagram with exact rows
		\begin{equation}\label{u-gl3}
			\begin{tikzcd}
				0 \arrow[r] & \mathfrak{gl}_3(\F_p) \arrow[r] \arrow[d,equal] & \tilde{U} \arrow[r] \arrow[d,hook] & U \arrow[r] \arrow[d,hook] &1 \\
				0 \arrow[r] & \mathfrak{gl}_3(\F_p) \arrow[r] & \on{GL}_3(\Z/p^2\Z) \arrow[r] & \on{GL}_3(\F_p) \arrow[r] & 1.
			\end{tikzcd}
		\end{equation}
		(The top row is in fact (\ref{restriction-alpha}).)    Let $V$ be a $p$-dimensional faithful representation of $U$ over $F$, and let $L/K$ be the generic Galois $U$-extension $F(V)/ F(V)^U$. Since $F$ contains a primitive $p$-th root of unity, by a theorem of Chu and Kang \cite[Theorem 1.6]{chu2001rationality} the field extension $K/F$ is purely transcendental, that is, $K\cong F(x_1,\dots,x_p)$, where the $x_i$ are algebraically independent over $F$.
		
		Embed $L/K$ into a separable closure of $K$, consider the corresponding continuous homomorphism $\rho_U\colon \Gamma_K\to U$, and write $\rho\colon \Gamma_K\to \on{GL}_3(\F_p)$ for the composite of $\rho_U$ and the inclusion $U\hookrightarrow \on{GL}_3(\F_p)$. We claim that $\rho$ does not lift to $\on{GL}_3(\Z/p^2\Z)$. Indeed, if $\rho$ lifts to $\on{GL}_3(\Z/p^2\Z)$, then $\rho_U$ lifts to $\tilde{U}$, that is, the pullback to $H^2(K,A)$ of the class of the top row of (\ref{u-gl3}) is trivial. Since $L/K$ is a generic Galois $U$-extension, by \Cref{connection} this implies that the top row of (\ref{u-gl3}) is negligible over $F$. As the top row of (\ref{u-gl3}) is the pullback of the bottom row and $U$ is a $p$-Sylow subgroup of $\on{GL}_3(\F_p)$, by \Cref{props}(4) the bottom row of (\ref{u-gl3}) is negligible over $F$, contradicting \Cref{thm-gl3}. (More simply, we could have noticed that the bulk of the proof of \Cref{thm-gl3} consisted in proving that the top row of (\ref{u-gl3}) is not negligible over $F$.) Thus $\rho$ does not lift to $\on{GL}_3(\Z/p^2\Z)$, as claimed. 
		
		For every $n\geq 3$, let $\rho_n\colon \Gamma_K\to \on{GL}_n(\F_p)$ be the composite of $\rho$ and the inclusion $\on{GL}_3(\F_p)\hookrightarrow \on{GL}_n(\F_p)$ as the top-left $3\times 3$ block. By \Cref{reduce-n-3}, for all $n\geq 3$, the homomorphism $\rho_n$ does not lift to $\on{GL}_n(\Z/p^2\Z)$.
	\end{proof}

\end{document}